\documentclass[11pt]{article}
\usepackage{amsmath}
\usepackage{amsfonts}
\usepackage{amssymb}
\usepackage{euscript}
\newenvironment{proof}{\noindent {\it Proof.~~}\ }{\  \rule{1mm}{2mm}\medskip}

\newenvironment{proofof}[2]{\noindent {\it Proof of #1}~#2: \
}{~\rule{1mm}{2mm}\medskip}
\newtheorem{theorem}{Theorem}
\newtheorem{lemma}[theorem]{Lemma}
\newtheorem{corollary}[theorem]{Corollary}
\newtheorem{proposition}[theorem]{Proposition}

\newtheorem{theirtheorem}{Theorem}

\newtheorem{theirlemma}[theirtheorem]{Lemma}

\def\Z{\mathbb Z}

\newcommand{\subgp}[1]{\langle{#1}\rangle}

\begin{document}
\title{A Structure Theory for  Small Sum Subsets}

\author{ Yahya O. Hamidoune\thanks{UPMC Univ Paris 06,
 E. Combinatoire, Case 189, 4 Place Jussieu,
75005 Paris, France,     {\tt hamidoune@math.jussieu.fr} }
}
\maketitle

\begin{abstract}
We develop a new method leading the structure  of finite  subsets $S$ and $T$ of an abelian group with $|S+T|\le  |S|+|T|$.
We show also how to recover the known results in this area in a relatively short space.
\end{abstract}


\section{Introduction}
Let  $A,B$ be finite subsets of $\Z/n\Z$ such that $|A|,|B|\ge 2+2s$ and   $|A+B|=|A|+|B|-1+s\le n-2-2s$.  For $n$ prime and
$s=0$, Vosper's Theorem \cite{tv} states that
  $A$ and $B$ are $r$--progressions, for some $r$.
 For $n$ prime and
$s=1$,  the authors of
\cite{hrodseth2} proved that there is an $r$ such that
 each of the sets $A$ and $B$ is obtained by deleting one element from an $r$--progression.  Some applications of the last  result may be found in literature.
In particular, it is used recently by  Nazarewicz, O'Brien,  O'Neill and  Staples in the characterization
of equality cases in Pollard's Theorem  \cite{naza}.
The authors of
\cite{hgochowla} obtained the description of the sets $A,B$ if  $s=1$,  $0\in B$ and if every element of $B\setminus \{0\}$ generates $\Z/n\Z$.

Kemperman's Structure  Theorem is  a deep classical result,
giving a recursive reconstruction for subsets $A,B$ of an abelian group with  $|A+B|=|A|+|B|-1$. A dual
equivalent reconstruction is given by Lev in \cite{levkemp}.
Recently  Grynkiewicz  obtained in \cite{davkem} a recursive reconstruction for the subsets $A,B$ of an abelian group with  $|A+B|=|A|+|B|$.

 Using  hyper-atoms and the strong isoperimetric property, the author
 obtained in \cite{hkemp} the description of the subgroups appearing in the reconstructions of Kemperman and Lev. In the present work, we investigate a more complicated hyper-atoms structure. The above
 mentioned results follow as corollaries, in a relatively short space, from one of our main theorems.
 Most of the  ingredients of our approach work for  $\mu<0$ and in
  the non-abelian case.  We need some terminology in order to present our results:

Let $S$ be a generating subset of $G$, with $0\in S$.
For a subset $X$, we put $\partial _S(X)=(X+S)\setminus X$ and $X^S=G\setminus (X+S)$.

 We  say that $S$ is {\em $k$--separable} if there is an $X$ such that
$ |X|\geq k$ and $|X^S|\geq
k$.

Suppose that $|G|\ge 2k-1.$
 The {\em $kth$--connectivity}
of $S$
 is defined  as

$$
{\kappa _k}(S)=\min  \{|\partial (X)|\   :  \ \
\infty >|X|\geq k \ {\rm and}\ |X^S|\ge k\},
$$
where $\min \emptyset =|G|-2k+1$.

 A finite subset $X$ of $G$ such that $|X|\ge k$,
$|X^S|\ge k$ and $|\partial (X)|={\kappa _k}(S)$ is
called a {\em $k$--fragment} of $S$. A $k$--fragment with minimal
cardinality is called a {\em $k$--atom}. We shall say that a subset $S$ is {\em degenerate} if there is a
 subgroup which is a $2$--fragment of $S$. A maximal subgroup which is a $2$-fragment
 of a {degenerate} subset $S$ will be called a {\em hyper-atom} of $S$.

The  basic facts from the isoperimetric method  may be found in
 \cite{hiso2007}.

 A subset of a group $G$ with cardinality $=1$ will be considered as a  $d$--progression  for every $d\in G$. A set $S$ will be called
 an {\em $(r,-j)$--progression} if it can be obtained from an arithmetic progression with difference $r$ by deleting $j$ elements. Notice that  an arithmetic progression $P$  of difference
$r$ is also an $(r,-j)$--progression if $r$ has an order  $\ge |P|+j$.  An $(r,-1)$--progression will be called sometimes a {\em near--$r$--progression}.

Let $H$ be a subgroup of an abelian group $G$ and let $d\in G/H$.
A set is said to be {\em $(H,-j)$--periodic} if it is obtained by deleting $j$ elements from a $H$--periodic set.
A partition  $A=\bigcup \limits_{i\in I}
A_i$  will be called
 a $H$--{\em decomposition}
of $A$ if for every $i$,  $A_i$ is the nonempty intersection of some $H$--coset
with $A$.
A $H$--decomposition $ X=\bigcup _{0\le i \le u}X_i$
such that $X_i+H+d=X_{i+1}+H$, for $1\le i \le u-1$, will be called a {\em $H$--progression} with difference $d$.

For a nonempty subset $X$ of $G$, we shall denote by $X^*$ an arbitrary translated copy of $X$ containing $0$.

The pair $\{S,T\}$ will be called an $H$--{\em essential pair} if
$ S=\bigcup _{0\le i \le u}S_i$
 and $T=\bigcup _{0\le i \le t}T_i$ are $H$--progressions  with the same difference such that $|S+H|-|S|=|T+H|-|T|=|H|$ and  one
 of the following  holds:
\begin{itemize}
\item [(i)] $|H|-1=|S_0|=|S_u|=|T_0|=|T_t|=1$.

 \item [(ii)] $|S_u|=|T_t|=1$, $|S_{u-1}|=|T_{t-1}|=|H|-1$,  $T_{t-1}+S_u=T_t+S_{u-1}$.

\item [(iii)] There are two subgroups $K_0,K_1$ of order $2$ such that
 $H=K_0\oplus K_1,$ $S_{0}^*=T_0^*=K_0$ and  $S_u^*=T_t^*=K_1$.
\end{itemize}
An essential pair with type (iii) will be called a {\em Klein pair}.


Our first goal is to prove the next  two results:

\begin{theorem}\label{twothird} Let  $ \mu \in \{ 0,1\}$. Let  $S$ be a  degenerate generating subset of an abelian group $G$ with $0\in S$ and let $H$ be a  hyper-atom of $S$.  Let $T$ be a  finite  subset of
 $G$  such that  $3-\mu \le |S|\le  \max(4-2\mu,|S|)\le|T|,$   $S+T$ is aperiodic and
  $ |S|+|T|-\mu=|S+T|\le  \frac{2|G|+2\mu}3.$
Then one of the following holds:
\begin{itemize}
\item [(i)]$\mu =0$ and  $|G|=3|S|=3|T|=4\kappa_2(T^*)=12.$
\item [(ii)] $\mu =0$ and $\{S,T\}$ is an $H$--essential pair. 
\item [(iii)]There are   $H$--progressions  $S=\bigcup \limits_{0\le i\le u}S_i$ and $T=\bigcup \limits_{0\le i\le t}T_i$  with a same difference
  such that  one of the sets $S\setminus S_u$, $T\setminus T_{t}$ is $H$--periodic and the other is $(H,-\nu)$--periodic, and $|T_{t}+S_u|=|T_{t}|+|S_u|-\nu-\mu,$ where $0 \le \nu \le 1-\mu$. Moreover $|T+H|-|T|\le |H|-\mu$.
\end{itemize}

\end{theorem}


\begin{theorem}\label{non=near}
Let $S$ be a finite  generating subset of an abelian group $G$ such
that $0 \in S,$  $3\le | S | \leq \frac{|G|+5\mu-4}{2}$.
Assume moreover that  $\kappa _{3-\mu} (S)\le |S|-\mu$ and that $\kappa _4 (S)\le |S|,$ if  $|S|=3=\mu+3.$
If   $S$ is  non-degenerate,  then $S$ is an $(r,\mu-1)$--progression for some $r$, where $0\le \mu \le 1$.
\end{theorem}

The organization of the paper is the following:

Section 2 presents our tools. Let $S$ and $T$ be finite subsets of an abelian group
 $G$  such that  $3-\mu \le |S|\le  \max(4-2\mu,|S|)\le|T|,$   $S+T$ is aperiodic and $ |S|+|T|-\mu=|S+T|,$ where $\mu \in \{0, 1\}$.
 In Section 3, assuming that $S$ is degenerate with a hyper-atom $H$ and that $|S+T|\le  \frac{2|G|+2\mu}3,$  we obtain a  $\frac{2n}3$--modular result asserting that for $|G|\neq 12$, $\phi(S)$ and $\phi(T)$
 are progressions with the same difference,  where $\phi :G\mapsto G/H$ denotes
the canonical morphism. In Section 4, we prove Theorem \ref{twothird}. In Section 5, we show that a subset $S$ with $ \kappa _{3-\mu}(S)\le |S|-\mu$ and $4\le |S|\le \frac{|G|+5\mu-4}2$
is either degenerate or a near-progression. In Section 6, we obtain a modular structure theorem encoding efficiently all the situations if $|S+T|\le  n-4$. We apply the last result in Section 7 to give the structure of  $S$ and $T$ allowing  $|S|=|T|=3$ and $ |S+T|=|G|- 3.$ We show how to recover  Structure Theorems of Kemperman \cite{kempacta} and Grynkiewicz\cite{davkem}.

We apply in the present work Kneser's Theorem (proved in less than two pages in \cite{tv}), Lemma \ref{chowla} (proved in few lines in \cite{hgochowla}) and Lemma \ref{vominus} (proved in few lines in \cite{hkemp}). We apply also  Theorem \ref{interfrag},  Theorem \ref{k=d} and Proposition \ref{strongip}
(these three results are proved in around two  pages in \cite{hiso2007}). We include short proofs for other needed lemmas,
making the work near self-contained.

We  omit  the easy case where $S+T$ is periodic (c.f. \cite{hkemp,davkem}), the trivial case $|S|=2$ and  the easy  case $|S+T|\le |G|-2$.

\section{Some tools}
\subsection{Preliminaries}

Let  $ A,B$ be finite subsets of an abelian group $ G $.  We write
$A+B=\{x+y \ : \ x\in A\  \mbox{and}\ y\in
  B\}.$ The subgroup generated by  $A$ will
be denoted by $\subgp{A}$.  Recall the following  results:
\begin{theirlemma}(folklore)
If
 $A$ and $B$  are  subsets of a finite group $G$
 such that $|A|+|B|\ge |G|+1$,
 then  $A+B=G$.
\label{prehistorical}
 \end{theirlemma}
 \begin{theirtheorem}\label{scherk}(Scherk's Theorem)\cite{scherck}
  Let $A$ and $B$ be  finite subsets of an abelian group $G$. If there is
an element   $c$   of  $G$ such
  that $|A\cap(c-B)| = 1$, then
    $|A + B| \geq |A| + |B| - 1.$
\end{theirtheorem}

  \begin{theirtheorem} (Kneser's Theorem) 
Let  $A, B$ be finite
subsets of  an abelian group.  If $A+B$ is aperiodic,  then $|A+B|\ge |A|+|B|-1$.

\end{theirtheorem}

\begin{theirlemma}\cite{hgochowla}\label{chowla} Let $0\in A$ be an  $(r,-1)$--progression  and let  $B\subset \subgp{A}
$ be  such that $\min (|B|, |A|)\geq 3$ and
$
|B+A| \leq |A| + |B|\leq |\subgp{A}|-4.
$
If $A+B$ is aperiodic, then  $B$ is an $(r,-1)$--progression.
\end{theirlemma}

\begin{theirlemma}\cite{balart} {Let  $X$ be a finite
 subset of an abelian group $G$. Then $X \subset (X^S)^{-S}$ and $(X^S)^{-S}+S=X+S$. \label{lee}}
\end{theirlemma}

Clearly $X\subset (X^S)^{-S}$.
Take   $x=y+s$, with $y\in (X^S)^{-S}$ and $s\in S$. We have  $x\in X+S$, otherwise $x-s\in X^S-S$  and hence
$y=x-s\notin (X^S)^{-S}$, a contradiction.$\Box$

We can use Kneser's Theorem to get  some isoperimetric duality:

\begin{lemma}{Let  $X$ be a
 subset of a finite abelian group $G$ such that $X+S$ is aperiodic and $|X+S|= |X|+|S|-\mu,$ where $0\le \mu $. Then
 $X^S-S$ is aperiodic. There is $0\le \zeta \le 1$ such that  $|X^S-S= |X^S|+|S|-\zeta$. Moreover $|(X^S)^{-S}|=  |X|+\zeta-\mu.$
\label{leee}}
\end{lemma}

\begin{proof}
By Lemma \ref{lee}, $X^S-S$ is aperiodic. By Kneser's Theorem, $\zeta \le 1$. Clearly $X^S-S\subset G\setminus X$, and  hence
\begin{eqnarray*}
|X^S|+|S|-\zeta +|(X^S)^{-S}|&=&|X^S-S|+|(X^S)^{-S}|= |G|\\&=&|X+S|+|X^S|= |X^S|+|X|+|S|-\mu.\end{eqnarray*}Thus $|(X^S)^{-S}|= |X|+\zeta-\mu$.\end{proof}

The following lemma is a very special case  of the main  result proved in  \cite{davkem}:
\begin{theirlemma}\cite{davkem}\label{3plus3}
Let $S,T$  be subsets of an abelian group $G$ with $ |S|=|T|=3,$ $S+T$ is aperiodic and
 $|T+S|=6-\mu,$ where $0\leq \mu \le 1.$  Then there exist $r,a\in G$   such that
  either one of the sets  $S$ and $T$ is an $r$--progression 
  or  $T=a+S$.
  \end{theirlemma}

\begin{proof} Without loss of generality, we may assume $0\in T\cap S$. Suppose that none of the sets  $S$ and $T$ is a progression.

Assume first that there is an $a\in S\setminus \{0\}$ with $2a=0$.  Put $H=\{0,a\}$ and $S=\{0,a,b\}.$
We have $|T+H|=2|H|,$ otherwise $T+S$ would contain a periodic subset of size $6$. By translating suitably $T$,
we may take $T=\{0,a,c\}$. Now $T+S\supset H\cup (b+H)\cup (c+H)$. We must have $b+H=c+H$ and hence $c=b+r,$ for some $r\in H$.
Thus $T=r+S$ and (2) holds. So we may assume that $2x\neq 0$ for every $x\in (S\cup T)\setminus \{0\}$.

Now for every $x\in T\setminus \{0\}$, we have $|(x+S)\cap X|\le 1$, otherwise $S$ would be
an $x$--progression, a contradiction.
Observe that $|(S+x)\cap (S+y)|=1$, for any two distinct elements $x,y\in T$, since otherwise putting $T=\{x,y,z\}$,
$|(S+z)\cap (S+x)|\ge 2$ or $|(S+z)\cap (S+y)|\ge 2$, a contradiction.

Notice that the last observation still valid if $S$ and $T$ are permuted.

Put $T=\{0,u,v\}.$  Since $|S\cap (S+u)|=1$, there is an  $a$ such that  $S-a=S_0=\{0,u,w\}$. Since $|S_0\cap (S_0+v)|=1$,
we have $w\in \{u+v, u-v, v,-v\}$. Up to a translation by $-u$, we have $w=v$ or $w=-v$.
 Assuming $w=-v,$ we have $S+T= \{0,u,v,2u,v+u,-v,u-v\}$. It follows that $u-v\in \{0,u,v,2u,v+u,-v,u-v\}$.
All possible cases imply that one of the sets  $S$ and $T$ is a progression, a contradiction.
Thus  $w=v$, and hence  (2) holds.\end{proof}

\subsection{Isoperimetric tools}

Let $S$ be a finite subset of an abelian group. A { $k$--fragment} of $S^*$ will be called a {\em $k$--fragment} of $S$.
This notion is independent on the choice of $S^*$  \cite{hejcv}. A $k$--fragment of $-S$ will be  called a {\em negative}
$k$--fragment of $S$.


\begin{lemma}\cite{hiso2007} \label{negative} Let $0\in S$ be a generating subset of an abelian group $G$. Let   $X$     be a  $k$--fragment of $S$ and let $A$ be a $k$--atom of $S$. Then
$-X$ is a negative $k$--fragment of $S$.
Moreover $X^S$ is a negative $k$--fragment of $S$ if  $G$ is finite.
In particular, $|X^S|\ge |A|.$
\end{lemma}

A fragment $X$ of $S$ such that $|X|\le |X^S|$ will be called a {\em proper} fragment.
The following result will be a  fundamental tool in this paper: 

\begin{theorem}\cite{hiso2007}
Let $0\in S$ be a generating subset of an abelian group $G$.
 If $X$ and $Y$ are two $k$--fragments of
$S$ such that $|X\cap Y|\ge k$ and $|(X\cup Y)+S|\le |G|-k$.
 Then  $X\cap Y$ and $X\cup Y$ are  $k$--fragments of $S$.

In particular, $X\cap Y$ is a $k$--fragment if $|X|\le |Y^S|$ or if $X$ and $Y$ are proper $k$--fragments.

\label{interfrag} \end{theorem}

The basic intersection theorem is the following:
\begin{theorem}\cite{halgebra,hiso2007}
Let $0\in S$ be a generating subset of an abelian group $G$.
 Let $A$ be a $k$--atom of $S$ and let
   $F$   be a   $k$-fragment of $S$ with  $|A\cap F|\ge k$. Then
  $A\subset F.$
In particular, $A=F$ if $F$ is a $k$-atom.
\label{interatoms} \end{theorem}
\begin{proof}
Suppose that $|A\cap F|\ge k$. By Lemma \ref{negative}, $|A^S|\ge |F|$.
By Theorem \ref{interfrag}, $A\cap F$ is a $k$--fragment and hence $A\cap F=A$.
\end{proof}

The structure of $1$--atoms is the following:
\begin{proposition} \label{Cay}\cite{hast,hiso2007} Let  $ S$  be a finite generating subset of an abelian  group $G$ with
$0\in S$.  Let $H$ be a $1$--atom of $S$ with $0\in H$ and let  $F$ be a $1$-fragment of $S$.
   Then
   $H$ is a subgroup and $F+H=F$. Moreover
 $\kappa _1(S)\geq \frac{|S|}{2}.$
\end{proposition}

 We need the following consequence of the above result:

\begin{proposition} \label{olsonreduction}
Let  $ Y$  be a finite subset   of an abelian  group $G$ with
$0\in Y$. Put $K=\subgp{Y}$ and let $Z\subset K$. Let  $ X$
be an aperiodic  subset   of  $G$ such that $|X+Y|\le |X|+r|Y|$ and
let   $X=X_0\cup \cdots \cup X_t$  be a $K$--decomposition. Set $W=\{i\in [0,t]: |X_i+Y|<|K\}$
and $P=\{i\in [0,t]: |X_i|=|K|\}.$  Then
\begin{itemize}
  \item[(i)] $|X+Y|\ge |X|+|W|\frac{|Y|}{2}$ and  $|W|\le 2r.$ If $|Y|\ge 3$  then $|W|\le 2r-1$.
  \item[(ii)] If $|W|=2r$, then  $|X+Y|= |X|+r|Y|$ and $|P|=t+1-2r$. Moreover $X_i$ and $Y$ are
progressions with a same difference, for every $i\in W$.
 \item[(iii)] If $|W|=2$ and $r=1$, then  $|(\bigcup _{i\in W} X_i)+Z|\ge |\bigcup _{i\in W} X_i|+|W|.$
 \item[(iv)] If $X+Y$ is aperiodic, $r=1$ and $W=\{w\}$, then  $X\setminus X_w$ is ($K$,-1)-periodic.
\end{itemize}
\end{proposition}

\begin{proof}
 Let $H$ be a $1$--atom of $Y$ with $0\in H$. We have by Proposition \ref{Cay},

 \begin{eqnarray*}
 |X|+r|Y|\ge |X+Y|&\ge& \sum_{ i \in W} |X_i+Y|+ \sum_{ i \notin W} |X_i+Y|\\
 &\ge& \sum_{ i \in W} (|X_i^*|+\kappa _1(Y))+ \sum_{ i \notin W} |K|\\
 &\ge &  \sum_{ i \in W}(|X_i|+ \frac{|Y|}{2})+  \sum_{ i \notin W}|X_i|\ge |X|+|W|\frac{|Y|}{2}.
  \end{eqnarray*}
 Hence $|W|\le 2r$. Assume now that $|W|= 2r$. Thus the  last chain consists of  equalities and therefore
 $  P=[0,t]\setminus W$ and  $\kappa _1(Y)=\frac{|Y|}{2}$.
 By Proposition \ref{Cay},
   $X_i+H=X_i$, for all $i\in W.$  In particular, $X+H=X$. Since $X$ is
   aperiodic, we have  $|H|=1$. Therefore $|Y|-1=|H+Y|-|H|=\kappa _1(Y)=\frac{|Y|}{2}$.
 Hence $|Y|=2$. Put $Y=\{0,d\}.$  Since  $|X_i+Y|=|X_i|+1,$  $X_i$ is a progression with
  difference as $d,$ for all  $i \in W$. Thus

$$ |\bigcup _{i\in W} (X_i+Z)|\ge \sum _{i\in W} |Z+X_i^*|
\ge \sum_{ i \in W}  (|X_i|+1)= \sum_{ i \in W}|X_i|+|W|,$$
since $d$  generates $K$ and  $|X_i|<|K|$, for all $i\in W$.

Suppose that $X+Y$ is aperiodic, $r=1$ and   $W=\{w\}$. By Kneser's Theorem,
$|X+Y|\ge t|K|+|X_w+Y|\ge t|K|+|X_w|+|Y|-1,$
and  (iv) holds.\end{proof}

We need the following description of $2$-atoms:

\begin{theorem} \cite{hacta}{  Let  $S$  be a finite
generating subset of an abelian group $G$ with $0\in S$ and $\kappa _2 (S) \leq |S|$.
Also assume that $|S|\neq |G|-6$ if $\kappa _2 (S) = |S|$.
 Let
 $0\in H$ be a $2$--atom of $S$. Then either  $H$ is  a subgroup  or  $|H|=2.$
\label{k=d} }
\end{theorem}

A short proof of Theorem \ref{k=d} is given in \cite{hiso2007}. A generalization of the above
 result to the case $\kappa _2 (S) \leq |S|+
  4$ is obtained by the authors of \cite{hgochowla}.
\subsection{Vosper subsets}

Let $0\in S$ be a  subset of an abelian group $G.$ We
shall say that $S$ is a {\em Vosper subset} if for all $X\subset G$
with $|X|\ge 2$, we have $|X+S|\ge \min (|G|-1,|X|+|S|)$.

We need the following lemma:
\begin{lemma}\cite{hkemp} { Let $S$ be a finite generating subset of an
abelian group $G$ such that $0 \in S$.   Let $X\subset G$  be such
that   $|X+S|=|X|+|S|-1$ and $|X|\ge |S|$. Assume moreover  that $S$ is either a Vosper subset
or a progression. Then for every $y\in
S$, we have $|X+(S\setminus \{y\})|\ge |X|+|S|-2$. \label{vominus}}
\end{lemma}

We need the following lemma which is a  consequence of Theorem \ref{k=d}:

\begin{proposition} \cite{hejcv,hiso2007} \label{ejcf}{
Let  $S$  be a  finite generating
subset of an
 abelian group $G$ such that $0\in S$, $|S|\leq (|G|+1)/2$ and $\kappa _2 (S) \leq |S|-1$.

 If  $S$ is not a progression then $S$ is degenerate.

                    }\end{proposition}


\begin{corollary} \label{haapv}
Let $S$ be a finite degenerate generating subset of an abelian group $G$ such
that $0 \in S$  and $\kappa _2 (S)\le
|S|<\frac{|G|}2$ and let $H$ be a hyper-atom of $S$. Then
$\phi (S)$ is either a progression or a Vosper
subset, where $\phi$ is the canonical morphism from $G$ onto
$G/H$.

\end{corollary}

\begin{proof}
Assume that $\phi(S)$ is neither a Vosper subset nor a progression.
Then $\kappa _2(\phi(S))\le |\phi (S)|-1$. We have $|S+H|\le |S|+|H|< \frac{|G|}2+|H|$. Therefore
$2|S+H|< |G|+2|H|$ and hence
$2|\phi (S)|\le \frac{|G|}{|H|}+1$.
By Proposition \ref{ejcf}, $\phi(S)$ has $2$--fragment $K$ which is a  subgroup.

We have $|\phi^{-1} (K)+S|\le (|K|+|\phi (S)|-1)|H|= |\phi^{-1} (K)|+|H+S|-|H|=|\phi^{-1} (K)|+\kappa _2(S).$
Since $K+\phi(S)\neq G/H$, we have $\phi^{-1} (K)+S\neq G.$ In particular, $\phi^{-1} (K)$ is $2$--fragment which is a subgroup, a contradiction.
\end{proof}

\subsection{ The strong isoperimetric property}
Let $S=\bigcup \limits_{0\le i\le
 u}S_i$ and $T=\bigcup \limits_{0\le i\le
 t}T_i$ be $H$--decompositions. A $(T,S,H)$--{\em matching}  is a family
$\{n_i; i\in J\},$ where $J\subset  [0,t]$ such that $G\setminus (T+H) \supset \bigcup \limits_{ i\in
 J}T_{i}+S_{n_i}$ is $H$--decomposition.
 We shall call $|J|$ the {\em size} of the matching.

We call the property in the   next result   the {\em
strong isoperimetric property}.

\begin{proposition} \cite{hiso2007}{ Let $G $ be an abelian group and let $S$ be a finite subset of $G$ with $0\in S$.
Let $H$ be a subgroup of $G$ which is a $2$--fragment
 and let   $S=S_0\cup \cdots \cup S_u$ and $T=T_0\cup \cdots \cup T_t$ be  $H$-decompositions.
If  ${|G|}\ge (t+u+1)|H|$, then there is a $(T,S,H)$--matching with size $=u$.
\label{strongip}}
\end{proposition}


\section{Modular progressions}
\begin{theorem}\label{modular}
Let  $S$ be a  degenerate generating subset of an abelian group $G$ with $0\in S$ and let $H$ be a  hyper-atom of $S$.
Let $\phi :G\mapsto G/H$ denotes
the canonical morphism.
Let $T$ be a  finite  subset of
 $G$  such that  $3-\mu \le |S|\le  \max(4-2\mu,|S|)\le|T|,$ $S+T$ is aperiodic and
  $|S+T|=|S|+|T|-\mu\le  \frac{2|G|+2\mu}3,$ where $0\le \mu \le 1$. Then  one of the following conditions  holds:
 \begin{itemize}
  \item [(i)] $\mu =0$ and $|G|=3|S|=3|T|=4\kappa_2(T^*)=12.$
       \item [(ii)]
  $|\phi(S+T)|=|\phi(S)|+|\phi(T)|-1$ and moreover  $\phi(S)$ and $\phi(T)$ are  progressions with the same
  difference.
 \end{itemize}
\end{theorem}

\begin{proof} Set $|G|=n$, $h=|H|$,
$|\phi(S)|=u+1$, $|\phi(T)|=t+1,$
$|\phi(S+T)|=k+1$ and
$q=\frac{n}{h}$. Take  $H$--decompositions  $T=\bigcup \limits_{0\le i\le t}T_i$ and  $S=\bigcup \limits_{0\le i\le u}S_i$ such
that  $|S_0|\ge \cdots \ge |S_u|$. For $0\le i \le u$, put $K_i=\subgp{S_i^*}.$
We shall also assume (by a suitable reordering) that $|K_0|\ge |K_u|$ in the case where $|S_0|=|S_u|$.
We have $|G|>|S+H|\ge 2|H|$, and hence $|G|\ge 6.$ Therefore $|T^S|\ge \frac{|G|-2\mu}3>1$.
 Since $|S+T| \le |G|-2,$ we have

\begin{equation}\label{connectiv}
uh=|H+S|-h=\kappa _{2}(S)\le |S|-\mu.
\end{equation}

It follows that for all $u\ge
j\geq 0$,
\begin{equation}\label{plein}
(j+1)|S_{u-j}|\ge |S_{u-j}|+\cdots +|S_u|\ge jh+\mu.
\end{equation}

It follows that for  $u\ge
2$
\begin{equation}\label{pleintrois}
|S_0|+|S_{u-1}|\ge\frac{2(|S_0|+|S_{u-1}|+|S_{u}|)}3\ge \frac{2|H|+\mu}3.
\end{equation}

Without loss of generality, we may assume that $0\in S_0$.

 Since $T$ is aperiodic, we have $(t+1)h> |T|\ge |S|\ge \kappa _2(S)=uh$, and hence

\begin{equation}\label{t>u}
t\ge u.
\end{equation}

Choose a (possibly empty) $(T,S,H)$--matching
$\{n_i, i\in J\},$ where $J\subset  [0,t].$ Put $|J|=r$.  Take a $H$--decomposition $S+T=\bigcup _{0\le i \le k}E_i$ such that
\begin{enumerate}
  \item $T_i+S_0\subset E_i$, for all $0\le i \le t$,
  \item $ \bigcup _{i\in J}(T_{i}+S_{ni}) \subset \bigcup _{1\le i \le r} E_{t+i}. $
\end{enumerate}
We also assume that   $ |E_{t+r+1}|\le \cdots \le |E_k| $ if $k\ge t+r+1$.
We shall choose the $H$--decomposition and  $J$  in order to maximize $(r,|E_{k}|)$ lexicographically.

We shall put
$P=\{i \in [0,k] :  |E_i|=h\}$ and $W=\{i \in [0,t] :  |E_i|<h\}.$

Suppose that $k\ge t+r+1$ and take an $s$ with $T_s+S_{\alpha s}\subset E_k$. Therefore $T_s+S_{ns}\subset E_j$, for some $t+1\le j\le t+r$, otherwise   $ J\cup \{s\} $ would give a matching with size $r+1$.
Since $1\le \min (n_s,\alpha s)$ and $n_s\neq \alpha s$, we have $u\ge 2$.
 Now we can choose $\alpha _s\ge u-1$,
otherwise  $ (J\setminus \{j\})\cup \{k\}$ gives a  matching contradicting our choice. In particular, \begin{equation} \label{equpsilon}
  |E_k|\ge |S_{u-1}| \ \mbox{and }\  u\ge 2,\ \mbox{ if}\  k\ge t+r+1. \end{equation}

{\bf Case 1}  $\phi(T)=G/H$.

Thus $t+1=q$.
Let us show that
$$u=1.$$
Suppose $u\ge 2.$ By (\ref{plein}), $|S_0|\ge \frac{uh+\mu}{u+1}\ge \frac{2h+\mu}{3}$. On the other side $$|T+S|
\ge \sum _{i\in [0,t]}|T_i+S_0|\ge q|S_0|\ge q\frac{2h+\mu}{3}\ge  \frac{2n+3\mu}{3}.$$
It follows also that $\mu=0$ and  $|T_i+S_0|=|S_0|=\frac{2h}{3}$, for all $i$. Also $u=2$ and $|S_2|=|S_1|=|S_0|$.
The same thing applies to $S_1$ and $S_2$, and hence $T_i^*+S_s=S_s$, for all $i,s$. Since $S$ is aperiodic
we must have $|T_i|=1,$ for all $i$. Since $T+S=T+S_0=T+S_1$, there are distinct elements $r,s$ with $T_r+S_0=T_s+S_1$.
It follows that $S_1=S_0+w,$ where  $\{w\}=T_r-T_s$. Now we have $T+S=T+S_1=T+S_0+w=T+S+w,$ a contradiction.

Assuming $K_0\neq H,$ we have by (\ref{plein}),
$h\ge 2|S_{0}|\ge |S_{0}|+|S_1|\ge h.$
Thus $$\frac{h}2=|S_{0}|\ge |K_0|\ge  |K_1|\ge |S_1|=\frac{h}2.$$
Thus we must have $S_0=K_0$ and $S_1=K_1+b$, for some $b$.
Since $S$ is aperiodic, we have   $K_0\cap K_1=\{0\}$ and hence $\frac{h^2}4\le h$. In particular, $|K_0|=|K_1|=2$ and $H$ is isomorphic to $K_0\oplus K_1$.
Since $t+1=q$,  we have  $S+T=(T+K_0)\cup (T+K_1+b)$. Then $E_i\supset (K_0+T_i)\cup (K_1+c),$ for some $c$, and hence $|E_i|\ge 3,$
for all $i$.
Therefore $|S+T|\ge 3q=\frac{3n}4>\frac{2n+2}3$, since $n\ge 3h= 12,$ a contradiction. Thus
 $$K_0= H.$$

Put $\rho =\max \{|H|-|T_i| ; i\in P\}.$
Since $|T_i+S_0|<h,$ for every $i\in W,$  we have
by Proposition \ref{Cay}
\begin{eqnarray}|T+S|
&\ge& \sum _{i\in P}|E_i|+\sum _{i\in W}|T_i^*+S_0|
\ge |P||H|+\sum _{i\in W}|T_i|+|W|\frac{|S_0|}{2} \label{-rho}\\
&\ge& |T|+\rho+|W|\frac{|S_0|}{2}.\label{eqrho}
\end{eqnarray}

Since $u=1$, $\phi(S_1)$ generates $G/H$. By a suitable translation of $T$,  we may assume the following:

\begin{enumerate}
  \item $0\in T_0$, and $|T_0|\ge \max \{|T_1|, \cdots ,|T_t|\}.$
  \item $\phi(T_i+S_1)=\phi (T_{i+1})$, for all $0\le i \le t-1$.
\end{enumerate}
Suppose that  $|W|\le 2.$  We have by (\ref{-rho}) and (\ref{plein})
\begin{eqnarray*}|T+S|
&\ge& |P|h+\sum _{i\in W}|T_i+S_0|\\&\ge& (q-|W|)h+|W|(\frac{h+\mu}2)
\\&\ge& (q-2)h+2(\frac{h+\mu}2)\ge \frac{2n}{3}+\mu
\end{eqnarray*}
Hence $q=3$, $\mu =0$ and $|S_0|= \frac{h}{2}$. It follows that $|S_1|= \frac{h}{2}$.  Also we have $|E_i|=|T_i+S_0|=|S_0|$, for all $i\in W$.
It follows that $T_i^*+S_0=S_0$, for all $i\in W$. Hence $|T_i|=1$, for all $i\in W$, since $T+S$ is aperiodic. Since $q=3$ and $|T|\ge 4$,
we have $|T_0|\ge 2$. Thus $P=\{0\}$. Therefore
$|T_0+S_1|=|E_1|=|S_0|=|S_1|$ and $S_1$ is periodic. Now $T+S\supset E_0\cup (T_0+S_1)\cup (T_1+S_1)$,
which a periodic subset of cardinality $\frac{2n}{3},$
a contradiction, proving that   $$|W|\ge 3.$$

Suppose that     $q\neq 3.$
 We must have $|P|=0$, since  otherwise there are  $p\in P$ and $s\in W$ with $T_p+S_1\subset E_s$. But
$h>|E_s|\ge |T_p+S_1|.$
By Lemma \ref{prehistorical}, $|T_p|+|S_1|\le h$, and hence $\rho \ge |S_1|.$
By (\ref{eqrho}), $|T+S|\ge |T|+|S_1|+3\frac{|S_0|}{2}>|T|+|S|,$ a contradiction.

By (\ref{eqrho}), $q=|W|=4.$ Since  $ |T+S_0|\le |T+S|\le |T|+2|S_0|,$ we have
by Proposition \ref{olsonreduction}, $|S_0|=2$ and
$|T+S_1|=|T|+4.$  Therefore $T+S_1=T+S=T+S_0$.
Thus $S_0$ and $S_1$ are aperiodic. Since $2\ge |S_0|\ge \frac{h}2$, we have $3\le h \le 4$. Since $|H|\le 4$, $S_0=S_1+e,$ for some $e$.
Now we have $T+S=T+S_0=T+S_1+e=T+S+e,$ a contradiction.
Therefore $$q=3.$$

We have  $\frac{n}{3}=h\le |S+H|-|H|= \kappa _2|S|\le |S|-\mu\le \frac{n}{3}-\frac{5\mu}6$. Hence  $$ \mu =0 \ \mbox{ and } \ |S|=|T|=h=\frac{n}{3}.$$

The next  step is to show that     $n= 12.$ Since   $7\le \frac{2n}3,$ we have  $n\ge 12$. Suppose that  $n>12$ and hence $h\ge 5$.
Put $L_i=\subgp{T_i^*}$, for $0\le i \le 2$.

Assume first that there is a $j$ with $L_j=H$.
By Theorem \ref{interfrag} and since $h=|T^S|$, the set $T_j=(w+H)\cap T$ (for some $w$) is a $2$--fragment.
In particular $|T_j^*+S|=|T_j^*|+|S|.$

Since $S$ is periodic and by Proposition \ref{olsonreduction}, applied with $Y=T_j^*$, $S_0$ and $S_1$ are  progressions with a same difference.
It follows, since $|S_0|\ge \frac{h}2>2$, that $$|T+S|\ge|T+S_0|= \sum _{i\le 3}|T_i+S_0|\ge |T|+3|S_0|-3\ge |T|+2|S_0|,$$ and hence $|S_0|=|S_1|=3$.
Since $S_0,S_1$ are progressions
 with the same difference and the same cardinality, we have $ S_1=S_0+b$, for some $b\neq 0$.
We have also $|T+S_1|\ge |T|+3|S_1|-3=|T+S|.$
Now we have $T+S=T+S_1=T+S_0+b=T+S+b,$ a contradiction.


 So we may assume that for all $j,$  $L_j\neq H$. Since $S_0$ generates $H$, we have
 \begin{equation}
2|T|\ge|T+S|\ge|T+S_0|
=\sum _{0\le i\le 2}|T_i+S_0|\ge \sum _{0\le i\le 2}
 2|T_i|= 2|T|.\label{EQTJ}
\end{equation}
Hence all the above inequalities are in fact equalities. In particular, $2|T_0|=|T_0+S_0|<h$,
since $0\in W$. We have also $T+S=T+S_0$.

Take an $L_i$--decomposition  $S_0=S_{i0}\cup S_{i1}$. Without loss of generality we may assume $0\in S_{i0}$ and $|S_{i0}| \ge |S_{i1}|$.
From the above inequalities, we have $T_i+S_{i0}=T_{i}$. Since $|S_{i0}|\ge \frac{|S_0|}2\ge \frac{h}4$ and since $|T_i+S_0|=2|T_i|$,
we see that $T_i$ is a single coset with cardinality $\in \{\frac{h}4,\frac{h}3\}$. Since $|T_0|+ |T_1|+ |T_2|=h$,
we have necessarily $|T_0|=|T_1|= |T_2|=\frac{h}3$. At least two of  subgroups $T_0,T_1^*,T_2^*$ have a non-zero intersection
(otherwise $h^3\le 27h$ and we get a contradiction), say $|T_0\cap T_1^*|\ge 2$ (the other cases being similar).

Observe that $|S_0|\le \frac{2h}3-1$, otherwise $S_0+T_0=S_0$, and hence $T+S=T+S_0$ would be periodic, a contradiction. In particular, $|S_1|>\frac{h}3.$

Now $T+S\supset (T_0+S_0)\cup (T_0+S_1)\cup  (T_1+S_1)$, which is  a periodic subset of cardinality $2h=|S+T|$, a contradiction.
So $$n= 12.$$

 Thus $|T_0|=|T_1|+1=|T_2|+1=2.$ Clearly
$T+S\supset (T_0+S_0)\cup (T_0+S_1)\cup (T_1+S_1)\cup (T_2+S_0)$. Since $|S+T|=8$, we have
 necessarily $T_1+S_1=T_2+S_0$. Thus $S_1=S_0+z$, where $\{z\}=T_2-T_1$. So we may write
$S=S_0+\{0,z\}$. Since $T+\{0,z\}$ involves three $H$--cosets and since $P=\emptyset$, we have
$|T+S|=|T+\{0,z\}+S_0|\ge |T+\{0,z\}|+3,$ forcing that $|T+\{0,z\}|=|T|+1$. Hence $\kappa _2(T^*)\le |T|-1$ (observe that $T^*$
generates $G$). We must have $\kappa _2(T^*)=3= |T|-1,$ otherwise $T$ would be periodic by Proposition \ref{Cay}.

{\bf Case 2}    $\phi(T)\neq G/H$, i.e.  $t+1<q$.

{\bf Claim 1} If  $u\ge 2$ then $|P\cap [0,t]|\ge 2.$

Suppose that   $u\ge 2$ and that  there is a $j\in [0,t]$ such that $P\cap [0,t]\subset \{j\}$ and put
$\delta =\max \{|E_i|; t+1\le i \le k\}$.

For all $i\neq j$, we have $|T_i|\le \frac{h}{3},$  since otherwise by (\ref{plein}) and  Lemma \ref{prehistorical}, $|S_0+T_i|=h$, a contradiction.
We have \begin{eqnarray}
|S|+|T|\ge |S+T|&\ge& \sum \limits_{i\in [0,t]\setminus \{j\}}|T_i+S_0|+ |T_j+S_0|+ \sum \limits_{i\in [t+1,k]}|E_i|\nonumber\\
&\ge & 2|T|-|T_j| +\delta+(k-t-1)|S_u|.\label{del}\end{eqnarray}

 Assume  $|T_j|> |S_u|$. By Lemma \ref{prehistorical}, $|S_i+T_j|=h$, for all $0\le i\le u$. Since $P\cap [0,t]\subset \{j\},$
  $\delta = h$ and $k\ge t+2$. By (\ref{del}), we have
 $|S|+|T|\ge |S+T|\ge 2|T|- |T_j|+ h+|S_u|> 2|T|,$
a contradiction, proving that $$|T_j|\le |S_u|.$$

Therefore and  by (\ref{del}), we have
 $|S|+|T|\ge |S+T|\ge 2|T|- |T_j|+ |S_u|.$
 It follows that  the last chain   consists of equalities and hence $T_j+S_0=T_j$ and therefore $|T_j|=h=|S_u|$. In particular, $S$ is periodic,
a contradiction.

Take a $2$--subset $R\subset |0,t]\cap P.$ Put $\gamma =\min \{|E_i| ; t< i < k\}$.  Now we have
\begin{eqnarray}
|S+T|&\ge& \sum \limits_{i\in  R} |E_i|+\sum \limits_{i\in [0,t]\setminus R} |T_i+S_0|+ \sum \limits_{i\in [t+1,k-1]} |E_{i}|+|E_{k}| \nonumber \\
&\ge &2h+(t-1) |S_0|+(k-t-1)\gamma+|S_{u}|.\label{r-}
\end{eqnarray}

{\bf Claim 2} $q\ge u+t+1$.

Suppose the contrary. Then $u\ge 2$.  By Lemma \ref{prehistorical}, $\phi(T+(S\setminus S_u))=G/H$. Hence   $k+1=q$ and $|E_i| \ge |S_{u-1}|$, for all $t+1\le i \le k$.
By (\ref{plein}) and (\ref{pleintrois})
\begin{eqnarray*}
|S+T|
&\ge& 2h+(t-1)|S_0|+(q-t-1)|S_{u-1}|\\
&=&  2h+(2t-q)|S_0|+(q-t-1)(|S_0|+|S_{u-1}|)\\
&\ge& 2h+ (2t-q)\frac{2h}3+\frac{4h(q-t-1)}3=\frac{2n+2h}{3},
\end{eqnarray*}
observing that $q\le  t+u\le 2t,$ a contradiction.

{\bf Claim } 3. $|\phi(S+T)|=|\phi(S)|+|\phi(T)|-1$.

Put $\beta =1$ if $k>r+t$ and $\beta =0$ otherwise.  We have
\begin{eqnarray}
|S|+|T|&\ge&|S+T|= \sum _{0\le i \le k}|E_i|\nonumber\\ &\ge& \sum _{i\in  [0,t]\setminus J}|T_i+S_0|+\sum _{i\in J}|T_i+S_{ni}|+\sum _{i\in  J}|T_i+S_0|+\beta |E_k|\nonumber\\
&\ge& |T|+ r|S_0|+ \beta |S_{u-1}| \label{AP2}
\end{eqnarray}

By Claim 2 and Proposition \ref{strongip}, $r\ge u.$ Suppose  that $u<r$.
By (\ref{AP2}),  $|S|=(u+1)|S_0|$. We have also
$T_i+S_0=T_i$, and hence $|T_i|=h,$  for all $i\in [0,t]\setminus J$.
Also  $T_i+S_{ni}=T_i,$  for all $i\in J$.

Since  $T$ is aperiodic, we have  by (\ref{plein}), $\frac{h}2\le |S_0|=|S_u|\le \frac{h}2$ and $u=1$. Take $l\in J$. It follows also that  $T_l+S_1=T_l$ and $S_0=S_0+T_l=S_0+T_l+S_1$, a contradiction since a generating set with size $<h$ can not have a period with size $\frac{h}2.$
So $r=u$.
Let us show that $k=r+u$. Suppose the contrary.
   By (\ref{equpsilon}), $|E_k|\ge |S_{u-1}|$ and  $u\ge 2$.
By (\ref{AP2}),  $|S|=u|S_0|+|S_{u-1}|$.
By (\ref{plein}),  $\frac{2h}3\le |S_0|=|S_u|$. In particular $|K_i|=h$, for all $i$. By (\ref{AP2}), we have also
$T_i+S_0=T_i$, and hence $|T_i|=h,$  for all $i\in [0,t]\setminus J$,
  $T_i+S_{ni}=T_i$ and hence $|T_i|=h,$ for all $i\in  J$, a contradiction since $T$ is aperiodic.
Thus $k=t+u$.

{\bf Claim 4}  If $u\ge 2$ then $ k\le q-3$.

Since $|\phi (T+S)|=k+1=t+1+u$, we have by Lemma \ref{vominus} that $|\phi (T+(S\setminus S_u)|\ge t+u$, and hence
$\gamma \ge |S_{u-1}|$.

Suppose that $t+u=k \ge q-2.$
  Then $2t+1\ge t+1+u=k+1\ge q-1$. 
 We have by (\ref{r-}), (\ref{plein}) and (\ref{pleintrois})
\begin{eqnarray*}
|S+T|&=& \sum _{i\in R }|E_i|+\sum _{i\in [0,t]\setminus R }|E_i|+\sum _{i\in [t+1,k-1]}|E_i|+|E_k|\nonumber\\
&\ge & 2h+(t-1) |S_0|+(k-t-1)|S_{u-1}|+|S_u|\\
&=& 2h+|S_0|+|S_{u}|+|S_{u-1}|+(2t-k)|S_0|+(k-t-2)(|S_{u-1}|+|S_{0}|)\\
&\ge& 4h+\mu+(2t-k)|S_0|+(k-t-2)\frac{4h}3\\
&=&\frac{2h(k+2)}3+\mu \ge  \mu+\frac{2n}{3},
\end{eqnarray*}
observing that $k= t+u\le 2t.$  It follows that $\mu =0$ and that the last chain consists of equalities. In particular $|S_0|+|S_{u-1}|=\frac{4h}3$
and $|S_0|+|S_{u-1}|+|S_{u}|=2h.$ Hence $|S_{0}|=|S_{u-1}|=|S_{u}|=\frac{2h}3$. It follows also that $|E_{i}|=\frac{2h}3$,
for all $i\in [0,k]\setminus R$. It follows that for all $i$, $|T_i|\le \frac{h}{3},$  since otherwise by  Lemma \ref{prehistorical}, $|S_0+T_i|=|S_1+T_i|=|S_2+T_i|=h$, a contradiction. Thus  $$
|S|+|T|\ge |S+T|\ge \sum \limits_{i\in [0,t]}|T_i+S_0|+ |E_k|
\ge  2|T|+|S_u|,$$ a contradiction.

Assume that  $u\ge 2$. By Claim 4, $q-2 \ge |\phi (S+T)|,$
and hence by  Claim 3, $\phi (S)$ is not a Vosper subset.  By Corollary \ref{haapv}, $\phi(S)$ is a progression for $u\ge 2$.
But  $\phi(S)$   is obviously a progression for $u=1$.

By Claim 3,  $\phi(T)$ is a progression with the same difference as $\phi(S)$
 if $t+1+u=|\phi (S+T)|\le q-1.$ Assume
that $q=t+1+u.$ By Claim 4,  $u=1$ and  $|\phi (T)|=q-1$. Thus $\phi (T)$ is a progression  with arbitrary difference.\end{proof}

\section{The $\frac{2n}{3}$--Theorem}

We start by a lemma converting modular structure into subsets structure.

\begin{lemma}\label{modstr}
Let  $S$ be a  generating subset of an abelian group $G$ with $0\in S$ and let $H$ be a subgroup such that $|S+H|-|H|\le |S|-\mu,$
 where $0\le \mu \le 1$. Let $\phi :G\mapsto G/H$ denotes
the canonical morphism.
  Let $T$ be a  finite  subset of
 $G$  such that  $3-\mu \le |S|\le  \max(4-2\mu,|S|)\le|T|,$  $S+T$ is aperiodic and
  $ |S|+|T|-\mu=|S+T|\le |G|-4+2\mu $. If $\phi(S)$ and $\phi(T)$ are  progressions with a same difference and if $|\phi(G)|\ge |\phi(S)|+|\phi(T)|-1$
then there are $H$--progressions  $S=\bigcup \limits_{0\le i\le u}S_i$ and $T=\bigcup \limits_{0\le i\le t}T_i$  with a same difference
  such that one of the following conditions holds:
\begin{itemize}
\item [(i)] $\mu =0$ and $\{S,T\}$ is an $H$--essential pair.
\item [(ii)] One of the sets $S\setminus S_u$ and $T\setminus T_{t}$ is $H$--periodic and the other is $(H,-\nu)$--periodic. Moreover $|T_{t}+S_u|=|T_{t}|+|S_u|-\nu-\mu,$ where $0 \le \nu \le 1-\mu$.
\end{itemize}
If $G$ is finite, then $|\phi(T^S)|+|\phi(S)|\le |\phi(G)|+1.$
Moreover if $|T^S-S|\le |T^S|+|S|-1$ then $\phi(R)$ and $\phi(S)$ are progressions with the same difference,
  for every subset $R\subset T^S,$ with $|R|\ge |T^S|-1.$
\end{lemma}

\begin{proof}
Take  $H$--progressions $S=\bigcup \limits_{0\le i\le u}S_i$ and $ T=\bigcup \limits_{0\le i\le t}T_i$  with a same difference $d_0$.
Set $K_i=\subgp{S_i^*}$, for $0\le i \le u$.
By  suitable translation and choice of $d_0$,  we may assume that $0\in S_0$,  $|S_0|\ge  |S_u|$ and that $|K_0|\ge |{K_u}|$ if $|S_0|=|S_u|$. For  $U\subset [0,u]$, we have
$ |U||H|- \sum _{i \in U} |S_i| \le |S+H|-|S|\le  |H|-\mu$.
 Thus
\begin{equation}\label{plein1}
\sum _{i \in U} |S_i|\ge (|U|-1) |H|+\mu.
\end{equation}
Take a $H$--decomposition $S+T=\bigcup _{0\le i \le k}E_i$ such that
\begin{enumerate}
  \item $T_i+S_0\subset E_i$, for all $0\le i \le t$;
   \item $T_{t}+S_{i}\subset E_{t+i}$, for all $1\le i \le u$.
\end{enumerate}
Set $P=\{i: |E_i|=|H|\}$ and $W=[0,t]\setminus P$.
Put $h=|H|$ and  $n=|G|=qh$.

{\bf Case 1}  $K_0\neq H$. By (\ref{plein1}), we have $\mu =0$,  $|K_0|=|K_u|=\frac{h}{2}$
and  \begin{equation} \label{noex}|S_i|=h, \ \mbox{for all } \ i\in [1,u-1]. \end{equation}
 Since $S$ is aperiodic, we have $K_0\cap K_u=\{0\}$, and hence  $h\in \{2,4\}$.

{\bf Subcase 1.1}  $h=2$.

We have $|E_i| \ge \max \{ |S_1|, \cdots ,|S_{u-1}|\}=h$, for all $1\le i\le t+u-1$.
Hence \begin{eqnarray*}
|T|+|S|&=&  |T_{0}+S_0|+(t+u-1)h+|T_t+S_u|=|S|+|T_{0}|+(t-1)h+|T_t|,
\end{eqnarray*}
Up to replacing $d_0$ by $-d_0$,
we may assume that  $|T_0|\ge |T_t|$.
Since $T$ is aperiodic we must have $|T_t|=1$.
If $|T_0|=1$, then  $\{S,T\}$ is an $H$--essential pair.  If $|T_0|=2$, then  (ii) holds with $\nu =1$.

{\bf Subcase 1.2}  $h=4$. Since $S$ is aperiodic, we have $
K_0\cap K_u=\{0\}$, and hence $H=K_0\oplus K_u$.
It follows that $|S_i|=h, \ \mbox{for all } \ i\in [1,u-1].$
One may check as in Subcase 1.1, that $|T_i|=h, \ \mbox{for all } \ i\in [1,t-1],$
$T_0^*=K_0$ and that
$T_t^*=K_1.$ Thus $\{S,T\}$ is an $H$--essential pair (a Klein pair).

{\bf Case 2}   $S_0$ generates $H$.

Observe that $|T+S|\ge |T+S_0|+\sum _{1\le i \le u}|E_{t+i}|\ge |T+S_0|+|S|-|S_0|$. Therefore $|T+S_0|\le |T|+|S_0|$.
Assume that  $|W|\ge 2.$ By Proposition \ref{olsonreduction}, $\mu =0,$ $|T+S_0|=|T|+|S_0|$, $|S_0|=2$, $|W|=2$ and $|T_i|=h$, for all $i\notin W$.
Moreover $T_i$ is a progression with a same difference as $S_0,$  for every $i\in W$.
Observe that for every  $i\in  W\setminus \{0\}$, we have  $i-1\in W,$ otherwise  $|H|=|T_{i-1}|$ and  $|H|=|T_{i-1}+S_1|\le |E_i|,$ a contradiction.
So $W=\{0,1\}$. 
We must have $t=1$, otherwise $|E_i|\ge \min (|T_i|,|T_t|)=h$, for all $i\ge 2.$ It follows that $|S+T|\ge (t+u-1)h+|T_0+S_0|+|T_1+S_0|\ge |T|+2+uh$ and hence $|S_u|=h$, a contradiction.
Since $T_0$ and $T_1$ are   progression with a same difference as $S_0,$ we have
 \begin{eqnarray*}|S+T|&\ge& |T_0+S_0|+|T_1+S_0|+|T_1+S_1|\\&\ge& |T_0|+|S_0|-1+|T_1|+|S_0|-1+|S_1|+|T_1|-1\\&\ge& |T|+|S|+|T_1|-1.
\end{eqnarray*}
Therefore $|T_1|=1$. Since $|S_0|\ge \frac{h}2$, we have $3\le h \le 4$. Since $|T_0+S_0|<h$, we have $|T_0|\le 2$. Therefore  $|T|\le 3$, a contradiction.
Thus  $|W|\le 1.$

Take an $r\in [0,t]$ such that $W\subset \{r\}$. We have
\begin{eqnarray*}
|S|+|T|&=& |E_r|+\sum _{ i\in [0,t]\setminus \{r\}}|E_i|+ \sum _{1\le i\le u}|E_{t+i}|\\
&\ge&|T_r+S_{0}|+th+  \sum _{1\le i\le u}|S_{i}+T_t|\\
&\ge & |S_{0}|+th+ \sum _{1\le i\le u}|S_{i}|
\end{eqnarray*}
 Hence for some $\epsilon \ge 0,$ we have
\begin{equation} \label{epsilo}
|T+H| -|H|\ge |T|+\epsilon.
\end{equation}

{\bf Subcase 2.1} $\epsilon = 0$.

Then the last chain consists of equalities.
In particular $E_r=S_0+T_r^*=S_0$ and $ |E_{t+i}|=|S_i+T_t|= |S_i|$, for all $1\le i \le u$.
 Let us show that \begin{equation}S_0+T_t^*=S_0.\label{pers0}\end{equation}
Assuming the contrary, we have necessarily $r\neq t$ and $|S_0|<h$. Since $S_0$ generates $H$, the period of $S_0$ has order
$<\frac{h}2$. In particular $|T_r|<\frac{h}2$. By (\ref{epsilo}), $|T_t|>\frac{h}2$, and hence $K_t=H$. Since $S_u+T_t=S_u$, we have
$|S_0|\ge |S_u|=h$, a contradiction.

It follows that $S+T_t=S$ and hence $|T_t|=1$. By (\ref{epsilo}), $|T_i|\ge h-1$, and hence $|E_i|\ge
|T_i+S_0|\ge h$, for $i\neq t$. Thus $r=t$ if $W\neq \emptyset$.

Assume first that  $W\neq \emptyset$. Thus  $W=\{t\}$.
We must have $|S_1|=1,$  otherwise $|E_t|\ge |S_1+T_{t-1}|\ge h,$  by Lemma \ref{prehistorical}, a contradiction.
We have $u=1$, otherwise $|S_0|+|S_1|+|S_u|\ge 2h.$ Thus $|S_0|+|S_1|\ge 2|H|-1$. Since $|S_0|\ge |S_u|$, we have $|S_0|=|H|$ and hence $|E_t|=h,$
a contradiction.
Hence $\{S,T\}$ is an essential pair.

Assume now that  $W=\emptyset$.
We must have $|S_u|=1$, otherwise $|E_i|\ge h,$ for all $0\le i \le t+u-1$ and hence  $|S+T|\ge (t+u)h+|S_u|>|S|+|T|$, since $\epsilon=\mu=0$.
We must have $|T_{t-1}|=h-1$, otherwise $|E_i|\ge h$, for all $t+1\le i \le t+u-1$ and hence  $|S+T|\ge (t+u)h+|S_u|>|S|+|T|$, since $\epsilon=\mu=0$.
Similarly  $|S_{u-1}|=h-1$. Since $\epsilon=\mu=0$, we have $|S_i|=h,$ for all $0\le i \le u-2$ and $|T_i|=h,$ for all $0\le i \le t-2$.
Therefore $\{S,T\}$ is an essential pair.

{\bf Subcase 2.2} $\epsilon \ge 1$.

By (\ref{epsilo}) for all $v<w$,
\begin{equation} \label{plein2}
|T_{v}| +|T_w|\ge |H|+\epsilon \ge |H|+1.
\end{equation}

Take $1\le r \le u-1$. Clearly $E_{t+r}\supset (T_t+S_r)\cup (T_{t-1}+S_{r+1}).$
By (\ref{plein1}) and (\ref{plein2}), we have $|T_t|+|S_r|+|T_{t-1}|+|S_{r-1}|\ge 2h+1.$
Then either $|T_t|+|S_r|>h$ or $|T_{t-1}|+|S_{r+1}|>h.$
 By Lemma \ref{prehistorical}, $ |E_{t+r}|= h.$ Therefore $[0,t+u-1]\subset P.$

 Put $\nu = (u+t)h-|S\setminus S_u|-|T\setminus T_t|$.
 We have now
\begin{eqnarray*}
|S|+|T|-\mu &=&|S+T|\\
&=& \sum _{0\le i\le t+u} |E_i|+ |T_t+S_u|\\
&=& (t+u)h+ |T_t+S_u|=|S\setminus S_u|+|T\setminus T_t|+\nu+|S_t+T_u|.
\end{eqnarray*}
Therefore $|T_t+S_u|=|T_t|+|S_u|-\mu -\nu$.
Since $[0,t+u-1]\subset P,$ $S_t+T_u$ is aperiodic. By Kneser's Theorem $\mu+\nu \le 1$.

Suppose  now that $G$ is finite.
 Since $T+S$ involves full cosets except for the extremities,  $\phi(T^S)$ is  a progression with the same difference as $\phi(S).$
By Lemma \ref{modstr}, $S\setminus S_u$ and  $T\setminus T_{t}$ are $(\mu-1)$--periodic.

  Assume that  $|\phi(T^S)|+|\phi(S)|\ge q+2$. Clearly there is a $v$ such  that  $S\setminus (S_u\cup S_v)$ is periodic and $|S_v|\ge |H|-2.$ Thus
 $|\phi(T^S)|+|\phi(S\setminus (S_u\cup S_v))|\ge q-1$. Thus $|G|-|T|\ge |T^S-S|\ge |T^S-(S\setminus (S_u\cup S_v))|+|S_v|\ge (q-1)|H|+|H|-2= |G|-2$, a contradiction. Suppose that $|T^S-S|\le |T^S|+|S|-1.$ It follows that $\kappa _2(S)\le |S|-1$ and hence $|H+S|-|S|<|H|$. Therefore $\{S,T\}$ is not an
 elementary pair. Thus (ii) holds. In particular
$(S+T)\setminus (S_u+T_u)$ is periodic. Hence $\phi(R)$ is a progression with the same difference as $\phi(S).$\end{proof}

\begin{proofof}{Theorem}{\ref{twothird}}
Suppose that  (i) does not  hold.
By  Theorem \ref{modular}, $\phi (S), \phi (T)$ are  progressions with the same difference and $|\phi (S+T)|=|\phi (S)|+|\phi (T)|-1.$
Take  $H$--progressions $S=\bigcup \limits_{0\le i\le u}S_i$ and $ T=\bigcup \limits_{0\le i\le t}T_i$  with a same difference.
 Since  $S$ is degenerate,  $|G|>2|H|$ and hence $|G|\ge 6.$
 Therefore $|S+T|\le \frac{2|G|+2}3<|G|-1,$
and thus
$u|H|=|H+S|-|H|=\kappa _{2}(S)\le |S|-\mu.$
By Lemma \ref{modstr}, (ii) holds or (iii) holds.\end{proofof}

\section{The non-degenerate case}
\subsection{Few lemmas}

Suppose that $0\in S$ and that $G=\subgp{S}$. Then clearly
 $1\le \kappa_1(S) \le \cdots \le \kappa _k(S)$.
If $\kappa _k (S) =
 \kappa _{k-1}(S)$,
 then  every $k$--fragment is a
 $(k-1)$--fragment. Also every $(k-1)$--fragment $F$, with $k\le |F|$ and $k \le |F^S|$ is a $k$--fragment.

The above trivial observation will be used extensively in this section.

Our strategy consists in replacing a set with its $3$--atom or $4$--atom. We need to show that
non-degeneracy is preserved by this operation.
\begin{lemma}\label{aperiodicatom}
Let $S$ be a finite generating non-degenerate  subset of an abelian group $G$ such
that $0 \in S$, $\kappa _2 (S)=|S|\le \frac{|G|-4}2.$
Let $0\in F$ be a $2$--fragment of $S$ with $|F|\ge 3$ and $|F^S|\ge 4$. Then
\begin{itemize}
\item[(i)] A proper $2$--fragment of  $S$ contains no nonzero coset,
\item[(ii)] $F+S$ is  aperiodic and $F$ generates $G$,
  \item[(iii)] Assume that  $|F|\le 4$ and that $|F|+|S|>6$. Then
 $F$ is non-degenerate.
 \item[(iv)] If $A$ is a $3$--atom of $S$
with $|A|\ge 4,$  then $|A|=4$ and $\kappa _2(A)=|A|$.
 \end{itemize}
\end{lemma}

\begin{proof} Suppose that (i) is false and take a minimal  proper $2$--fragment  $X$ containing a nonzero subgroup $Q$.
Take a  $y\in  Q$.  We have
 $ |(X+y)\cap X|\ge |Q |\ge 2.$  By  Theorem \ref{interfrag}, $(X+y)\cap X$ is a $2$--fragment (clearly a proper one).
 By the minimality of $X$, we have $X=X+y$. Therefore $X+Q=X$.
 Since $X$ is not a subgroup, there is an $x$ with  $x+ X\neq X$. Observe that $X\cap (x+X)$
is $Q$--periodic. We have
 $ |(X+x)\cap X|\ge |Q +x|\ge 2.$  By  Theorem \ref{interfrag}, $(X+x)\cap X$ is a $2$--fragment.
By the minimality of $X$, we have $(X+x)\cap X=X$, and hence  $X+x=X$, a contradiction. This proves (i).

Let us show that $F+S$ is aperiodic. Suppose that $F+S+Q=F+S$, for some nonzero subgroup $Q$.
By the definition of $\kappa _2,$ we have $$|F|+|S|=|F+Q+S|\ge |F+Q|+\kappa _2(S)=|F+Q|+|S|.$$ It follows that $F=F+Q$ is periodic.
 By (i), $ |F|>|F^S|$. By Lemma \ref{negative},
 $-F^S=G\setminus (F+S)$ is a proper  periodic
$2$--fragment, a contradiction.

 Put $N=\subgp{F}$ and $s|N|= |S+N|$. Assume that $s>1$.
 By Proposition \ref{olsonreduction},(iv),  $ \frac{|G|}2> |S|\ge (s-1)|N|$. If follows that $S+N\neq G$ and that $|N+S|-|N|\le |S|=\kappa _2(S)$. Thus $N$ is a $2$--fragment of $S$, a  contradiction.

Suppose that (iii) is false.
Since $|G|\ge |F|+|S|+4\ge 11$ and $|G|$ is composite,
we have  $|F+S|\le \frac{|G|-4}2+4\le \frac{2|G|}3.$

By   Theorem \ref{twothird},   $|S+H|-|H|\le |S|$. Thus $H$ is a $2$--fragment of $S$,
 a contradiction.

 Clearly, we may assume that $0\in A$.  By (ii), $A$ is aperiodic and $A$
 generates $G$. Since $|A+S|\le |A|+|S|$, we have $\kappa _2(A)\le |A|.$
 Let $H$ be a $2$--atom of $A.$

 Suppose that $|A|\ge 5.$
 Assume first that $|H|>2$  and take a $3$--subset $\{0,z,z'\}$ of $ H.$
 By Theorem \ref{interatoms}, $|A\cap (A+x)|\le 2,$ for every $x\neq 0$. Thus 
 \begin{eqnarray*} \kappa _2(A)+|A|&\ge& |H|+|A|\ge 2+\kappa _2(A)\\&=& |A+\{0,z,z'\}|\ge |A|+|A|-2+|A|-4=3|A|-6,\end{eqnarray*}
and hence  $$2|A|\le \kappa _2(A)+6.$$

 Suppose that $ \kappa _2(A)\le |A|-1.$ By Proposition \ref{Cay}, $|H|\le \kappa _2(A)\le |A|-1\le 4.$
 By
 Theorem \ref{k=d}, $H$ is a subgroup. Take a $H$--decomposition  $A=\bigcup \limits_{0\le i\le
 t}A_i$. Since $|H|\le 4$ and by (i), $|A_i|\le 2,$ for every $i$. Hence $u\ge 2$ and thus $6\le u|H|=\kappa _2(A),$ a contradiction, proving that 
 $ \kappa _2(A)= |A|.$ It follows that $S$ is $2$-fragment of $A$. Since $S$ is non-degenerate, there is an $r$ such that $\{0,r\}$ is a $2$--atom of $S$. Take a minimal $2$--fragment $R\subset S$ of $A$ such that $|\{0,r\}+R|=|R|+2$ and $|R|\ge 3 $ (note that $S$ is a such fragment).
 Clearly $|R\cup (r+R)|\le |G|-2.$ By Theorem \ref{interfrag}, $R\cap (r+R)$ is a $2$--fragment  of $A$ such that $|R\cap (r+R)|=|R|-2.$
 It follows that $|R|\le 4.$ Thus $|A|>|R|,$ a contradiction proving that
 $  |A|\le 4.$

  Thus $|H|=2,$
  say $H=\{0,z\},$ for some $z$. Since $A$ is aperiodic we have by Theorem \ref{interatoms}, $|A\cap (A+z)|\le 2$. Hence
$$ 2+|A|\ge 2+\kappa _2(A)= |A+\{0,z\}|\ge |A|+|A|-2=2|A|-2,$$
and hence  $|A|\le 4,$ a contradiction.

By (iii), $A$ is non-degenerate.  Assume that $\kappa _2 (A)\le |A|-1$. There is an $r$
such that $A$ is an $r$--progression by Proposition \ref{ejcf}, and hence $|A\cap (A+r)|=3,$ contradicting Theorem \ref{interatoms}.\end{proof}

We recall that the arcs of Cayley graphs defined on a group $G$ by a subset $S$
are usually colored by the elements of $S\setminus \{0\}$. It will be helpful  to have this image
in mind.  However we assume  no knowledge of Cayley graphs.

Put  $E=\{(x,y)\in A\times A : x-y\in S \setminus \{0\}\}$.
The family $\{x-y; x-y\in S \setminus \{0\}\}$ will be called the {\em family of  colors} present in $A$.

\begin{lemma}\label{colors}
Let $S$ be a finite  subset of an abelian group $G$ such
that $0 \in S$ and $\kappa _3(S)=|S|.$ Let $F$ be a $k$--fragment with $|F|\ge k+1$
and let $a\in F+S$ be such that $|(a-S)\cap F|=1,$ say $(a-S)\cap F=\{b\}$.
Then $F\setminus \{b\}$ is a $k$--fragment.

 Let $0\in A$ be a $k$--atom of $A$
with   $|A|\ge k+1$.
Put $S^\ast=S\setminus \{0\}$ and $E=\{(x,y)\in A\times A : x-y\in S^\ast\}$.
Then
\begin{itemize}
  \item[(1)] {For every} $x\in A+S,$ $|(x-S)\cap A|\ge 2$.
  \item[(2)]  $|A|\le |E|=\sum _{x\in A} |(x-S^\ast)\cap A|\le (|S|-1)|A|-2\kappa_k(S).$
  \item[(3)] There is a nonempty subset  $R\subset E$ {such that} $\sum _{(x,y)\in R} (x-y)=0.$
\end{itemize}

\end{lemma}

\begin{proof}
We have
$(F\setminus \{b\})+S\subset ((F+S)\setminus\{a\})\cup \{b\} $, and hence $F\setminus \{b\}$ is
 a $k$--fragment.

Bounding   the total number of arcs inside $A$ or reaching $\partial (A)$ from $A$ by the number of arcs leaving $A$, we have using (1),
$|A||S^\ast|\ge \sum _{a\in \partial (A)}|(a-S^\ast)\cap A|+ \sum _{a\in A}|(a-S^\ast)\cap A|\ge 2\kappa_k(S)+|A|$, and (2) follows.

 In the graph induced by $A$, every vertex receives an arc colored by an element of $S^\ast$, by (1).
Since  $A$ is finite, $A$ must contain a directed cycle, $R=\{(a_1,a_2),(a_2,a_3)  \cdots , (a_j, a_{j+1})\}$, with
$a_{j+1}=a_1$.
 We have 
 
 $\sum _{(x,y)\in R} (x-y)=\sum _{1\le i \le j} (a_{i+1}-a_i)=a_{j+1}-a_1=0$.\end{proof}

Now we prove the optimality of the $4$--atom of a subset of size $3$.
\begin{lemma}\label{34}
Let $S$ be a finite generating non-degenerate subset of an abelian group $G$ such
that $0 \in S,$  $ | S |=3$ and $\kappa _4 (S)=\kappa _2 (S)=|S|.$
Let $0\in A$ be a $4$--atom of $S$. Then $A$ is non-degenerate and $|A|=4$.
  \end{lemma}
\begin{proof}
We shall assume that $S=\{0,u,v\}$ is translated  in order to maximize the order of  $u-v$.
Suppose to the contrary that $|A|\ge 5$.  By Lemma \ref{colors},(1),
every element of $\partial (A)$ the colors $u$ and $v$. Thus $(\partial (A)+\{u,v\})\cap \partial(A)=\emptyset$.
Also we have  $\partial (A) \subset (A+u)\cap (A+v)$.
By Theorem \ref{interfrag},  $3\ge |(A+u)\cap (A+v)|\ge |\partial (A)|$, and hence $\partial (A) = (A+u)\cap (A+v)$.
By Theorem \ref{interfrag},  $\partial (A)$ is a $2$-fragment of $S$.
 It follows that  $\partial (A)+u=\partial(A)+v=\partial ((\partial (A))$.
Thus $u-v$ has order $=3$. By considering $S-u$ and $S-v$ and the minimality of the order of $u-v$, we see that
the orders of $u$ and $v$ are at most $3$. Since $S$ generates $G$, we have $|G|\le 9$, contradicting the $4$-separability of $S$.
 \end{proof}

\subsection{Proof of Theorem \ref{non=near}}

The case $\mu=1$ follows by Proposition \ref{ejcf}. So we may take $\mu=0$ and $\kappa _2(S)=|S|$.

{\bf Claim 1} If $|S|=4$ then $S$ is a near-progression.

Put $A_0=S$.
Let $A_1$ denotes a $3$--atom of $S$ and let $A_2$ denotes a $4$--atom of $A_1$ such that $0\in A_1\cap A_2$. Suppose that $\min(|A_1|,|A_2|)\ge 4$.
By Lemmas \ref{aperiodicatom},  $A_1$ and $A_2$ are non-degenerate generating subsets with $|A_1|=|A_2|=4,$ and $\kappa _2(A_1)=\kappa _2(A_2)=4.$
It follows that $A_0$ is a $3$--atom of $A_1$ (inducing a symmetry between $A_0$ and $A_1$). By Theorem \ref{k=d}, there is an $r$ such that $\{0,r\}$ is a $2$-atom of $A_1$ and hence $|A_1+\{0,r\}|=|A_1|+2.$ Therefore there is a $u$ such that $\{0,u\}+\{0,r\}\subset A_1-a,$ for some $a\in A_1$. By  suitably translating $A_1$, we may assume that $u\neq -r$ (otherwise we replace $A$ by $A+r$) and that $a=0.$ By the definition of $\kappa _2,$ for every  $x\in \{u,r\},$ we have $|A_2+\{0,x\}|\ge |A_2|+2.$ We must have $x=r,$
since otherwise $A_0$ contains two $r$--arcs and two $u$--arcs. But the total number of
arcs colored by elements of $A_1\setminus \{0\}$ is $4$, by Lemma \ref{colors}. Thus the sequence  $\{r,r,u,u\}$ represents  the family of colors inside $A_0.$
By Lemma \ref{colors}, there is  a nonempty subfamily $R$ summing to $0$.
 We have $|R|\neq 2$,  since  $2r\neq 0$ and $2u\neq 0$ by Lemma \ref{aperiodicatom}.
We have $|R|\neq 4$, since otherwise $2(u+a)=2u+2a=0$, and $S$ would contain the coset $\{0,a+b\}$, contradicting Lemma \ref{aperiodicatom}.
It follows that $|R|=3$. Without loss of generality we may assume $R=\{r,r,u\}.$
Therefore $2r+u=0$. Thus $A=\{0,u,-2u,-u\}$, and hence $A'=\{0,u,2u\}$ is a $3$-atom of $A$, a contradiction.
Thus $$\{0,r,2r\}\subset A_1.$$

Let us show that $|A_0+\{0,r\}|=|A_0|+2.$ Assuming the contrary, we have $|A_0+\{0,r\}+\{0,r\}|\le |A_0+\{0,r\}|+1.$
In particular $A_0+A_1$ is the union  of full $r$--cosets  and an $r$--progression. Thus
$$|A_0+A_1|+1 \ge  |A_0+A_1+\{0,r\}|=|A_0+\{0,r\}|+|A_1|\ge |A_0|+|A_1|+2,$$ a contradiction.
Thus $\{0,v\}+\{0,r\}\subset A_0-b,$ for some $b\in A_0$ and some $v$. As for $A_1,$ we see that we may take $b=0$
and $v=r$. It follows that $(2r-A_0)\cap A_1\supset \{0,r,2r\}.$ By Lemma \ref{colors},(1) and (2)
$5\le \sum_{x\in A_1} |(x-A_0^\ast)\cap A_1|\le 3|A_1|-2\kappa_2(A_0)=4,$ a contradiction.
Thus there is an $i\in \{0,1\}$ such that $A_i$ has a $3$--atom $M$ with $|M|=3$. Put $A_i=T.$
By Lemma \ref{aperiodicatom}, $M$ is non-degenerate and $\subgp{M^*}=G$.

Note that  $M$ is not a near-progression, otherwise by successive applications Lemma \ref{chowla}, we see that $S$ is a near-progression
and the Claim holds. By Proposition \ref{ejcf}, $\kappa _2(T)\ge |T|$. It follows that $T$ is a $2$--fragment of $M.$

Clearly $ \sum_{x\in T+M} |(x-M)\cap T|\le |T||M|=12$. Thus there is $c\in M+T$ such that $|(c-M)\cap T|=1$. By Lemma \ref{colors},
$T$ contains a $2$--fragment $F$ of $M$ with $|F|=3$. Observe that  $F$ is not a progression,
otherwise $M$ would be a near-progression.

By Lemma \ref{3plus3}, $M=F+z,$ for some $z.$ By translating suitably $M$, we may assume that $M\subset T$.
Now $8\le |T+T|\le |T+M|+1\le 2|T|=8.$

By Theorem \ref{k=d}, there is an $r$ such that $\{0,r\}$ is a $2$-atom of $T$ and hence $|T+\{0,r\}|=|T|+2.$ Therefore there is a $u$ such that $\{0,u\}+\{0,r\}\subset T-a',$ for some $a'\in T$. Without loss of generality, we ay assume that $a'=0$ and that $u\ne -r$(otherwise we replace $T$ by $T+u$).

{\bf Case 1} $u=r$.  Put
$T=\{0,r,2r,w\}.$ We have $T+T=\{0,r,2r,3r,4r\}\cup \{w,w+r,w+2r\}\cup \{2w\}$.  We can not have $2w\in \{w,w+r,w+2r\}$.
Then $2w\in \{0,r,2r,3r,4r\}$. We can not have $2w\in \{0,2r,4r\}$, otherwise $T$ would contain a non-zero coset contradicting  Lemma \ref{aperiodicatom}. Then $2w\in \{r,3r\}$.
If $2w=r$, then  $T=\{0,w,2w,4w\}$ a contradiction.  So  we must have $2w=3r$, and hence
 $T-r=\{-r,0,r,w-r,\}=\{-2(w-r),0,2(w-r),w-r\}$. Therefore  $T$ is a $(w-r,-1)$--progression, a contradiction.

{\bf Case 2} $u\neq r$ and hence $T=R\cup \{0\},$ where $R=\{u,v,u+v\}.$

One may see easily using Lemma \ref{aperiodicatom} that $R\cap (-T)=\emptyset$. It follows that $|R\cap (R+R)|\ge 2$.
Without loss of generality we may take $u\in R+R=(u+R)\cup (u+v+R)\cup (u+v+\{0,v\})\cup \{2v\}.$ Using Lemma \ref{aperiodicatom}, we see that $T$ is a near-progression, a contradiction.

We shall now prove the theorem:

Assume first that $|S|\ge 4$ and let $A$ be $3$-atom of $S$. If $|A|=4,$ then by Lemma \ref{aperiodicatom}, $\kappa _3(A)=|A|$.
By Claim 1, $A$ is a near-progression. By  Lemma \ref{chowla}, $S$ is a near-progression. Suppose the contrary. By Lemma \ref{aperiodicatom},
$|A|=3$. Clearly  $|G|\ge |A+S|+4\ge 12$. Let $A'$  denotes  a $4$-atom of $A$. By Lemmas \ref{aperiodicatom} and \ref{34}, $A'$
is non-degenerate and $|A'|=4.$ By Claim 1, $A'$ is a near-progression. By  Lemma \ref{chowla} applied twice, $A$ and $S$ are   near-progressions.

Assume $|S|=3$. Let $A'$  denotes  a $4$-atom of $S$. By Lemmas \ref{aperiodicatom} and \ref{34},$A'$
is non-degenerate and $|A'|=4.$ By Claim 1, $A'$ is a near-progression. By  Lemma \ref{chowla}, $S$  is  a near-progressions.\rule{1mm}{2mm}\medskip

\section{The ($n-4$)--modular Theorem}

We shall now describe the structure if $|S+T|\le |G|-4$.

Let $S$ be a finite subset of an abelian group $G$. A subgroup $H$ is said to be a {\em super-atom} of $S$  if either $H=\subgp{S^*}$ or
 $H$ is a hyper-atom of $\subgp{S^*}$.
\begin{theorem}\label{supermodular}
Let  $S$ and $T$  be finite   subsets of an abelian group $G$ generated by $S^*\cup T^*$.
Also assume  that   $3-\mu \le |S|\le  \max(4-2\mu,|S|)\le|T|,$ $S+T$ is aperiodic and
  that $|S+T|=|S|+|T|-\mu\le |G|-4+2\mu,$ where $0\le \mu \le 1$. Then  one of the following conditions  holds:
 \begin{itemize}
  \item [(i)]  $S$ and $T$ are $(r,\mu-1)$--progressions for some $r$.
       \item [(ii)] There is a subgroup $H$ such that
  $|\phi(S+T)|=|\phi(S)|+|\phi(T)|-1$ and moreover  $\phi(S)$ and $\phi(T)$ are  progressions with the same
  difference if $\min \{|\phi(S)|, |\phi(T)|,  |\phi(G)|-|\phi(S+T)|\}\ge 2$,
 \end{itemize}
  where    $\phi:G\mapsto G/H$ is the canonical map. Moreover $H$ is a super-atom of $S$ or $T^S$ if $|G|\neq 12.$
\end{theorem}

\begin{proof} Without loss of generality we may take $T=T^*$ and $S= S^*.$ Assume first that $S$ generates a  proper subgroup $K$ (  not containing $T$ necessarily) and let
$T=\bigcup \limits_{0\le i\le t}T_i$ be a $K$--decomposition. Put $W=\{i : |T_i+S|<|K|\}$.

By Proposition \ref{olsonreduction},  $W=\{v\},$ for some $v$.
Put $\nu = t|K|-|T\setminus T_v|$.
We have $$|T|+|S|-\mu=|T+S|=t|K|+|T_t+S|=|T\setminus T_t|+\nu+|T_t+S|.$$

Thus $|T_t+S|=|T_t|+|S|-\mu-\nu.$ Since $T+S$ is aperiodic, $T_t+S$ is aperiodic. By Kneser's Theorem,
$\mu+\nu \le 1.$ Hence (1) holds with $H=K$.

Assume now that $S$ generates $G$.
  If $S$ is a near-progression, the result holds by Lemma \ref{chowla}.
  So we may assume that  $S$ is not a   near-progression.
Put $X=T^S$ and $Y=(T^S)^{-S}=G\setminus (X-S).$  By Lemma \ref{leee}, $X-S$ is aperiodic and there is
$0\le \zeta \le 1$ with $|X-S|= |X|+|S|-\zeta$.

 {\bf Case 1} $|S|\le |T^S|$.
 We have $|S|\le  \frac{|T|+|S|}2\le  \frac{|T+S|+\mu}2\le \frac{|G|+\mu-4}2.$
  By Theorem \ref{non=near}, $S$ is  degenerate.
 Let $H$ be a hyper-atom of $S$ and put $q|H|=|G|$.

{\bf Subcase 1.1} $|T|\le |T^S|.$ Hence $|S|+|T|\le \frac{2|G|+2\mu}3$. The result holds by Theorem \ref{modular} unless $|G|=3|T|=12
 =4\kappa_2(T).$  By Proposition \ref{ejcf}, $T$ is degenerate.  The result holds by Theorem \ref{modular} with $H$ denoting a hyper-atom of $T$.

 {\bf Subcase 1.2} $|S|\le |T^S|< |T|$ and  hence $|S|+|T^S|\le \frac{2|G|+2\mu}3$.

In particular $|G|> 12$, if $|S|=4$.
By Lemma \ref{leee}, $X-S$ is aperiodic and  $|X-S|=  |X|+|S|-\zeta$.
By Theorem \ref{modular}, $\phi(S)$ and $\phi(X^{-S})$ are progressions with the same difference.
By Lemma \ref{modstr}, $|\phi(S)|+|\phi(X^{-S})|\le q+1.$ The result holds if $T=X^S$. Suppose the contrary. By Lemma \ref{leee},
 $\zeta =1$. Then  $|T|=|X^{-S}|-1$. By Lemma \ref{modstr},  $\phi(T),\phi(S)$ are progressions with a same difference.

 {\bf Case 2} $|T^S|< |S|$.

Assume first that $X^*$ generates a proper subgroup $Q$ and put $|G|=q'|Q|$. Take
$Q$-decompositions   $T=\bigcup \limits_{0\le i\le t}T_i$ and $S=\bigcup \limits_{0\le i\le u}S_i$.
Since $X$ is contained in a
single coset, say $X\subset T_t+S_u$, the other $Q$-cosets are all contained in $T+S$. By Theorem \ref{scherk}, we have
 $t+u+1\le q'$. Hence  $|T|+|S|-\mu=|T+S|\ge (t+u)|Q|+|T_t|+|S_u|-1$.
In this case (i) holds or (ii) holds.

Assume now that $X$ generates $G$. Since $|X-S|\le |X|+|S|,$ $X$ can not be a near-progression  by Lemma \ref{chowla}.
By Theorem \ref{non=near}, $X$ is degenerate.
Let $N$ be a hyper-atom of $X$ and let $\psi:G\mapsto G/N$ be the canonical morphism.
By Theorem \ref{modular}, $\psi(S)$ and $\psi(X^{-S})$ are progressions with the same difference.
Also $|\psi(S)+\psi(X^{-S})|\le q'+1.$ The result holds if $T=X^{-S}$. Suppose the contrary. By Lemma \ref{leee},
 $\zeta =1$. Then  $|T|=|X^{-S}|-1$. By Lemma \ref{modstr},  $\psi(T),\psi(S)$ with a same difference.\end{proof}

\section{The ($n-3$)--Structure Theorem}
\begin{theorem}\label{n-3}
Let  $S$ and $T$  be finite   subsets of an abelian group $G$ generated by $S^*\cup T^*$.
Assume moreover that   $3-\mu \le |S|  \le |T|,$ $S+T$ is aperiodic and
  that $|S+T|=|S|+|T|-\mu \le |G|-3-\mu,$ where $0\le \mu \le 1$. Then  one of the following conditions  holds:
 \begin{itemize}
 \item[(i)]$|S|=3=3+\mu$ and there is an $a$ such that  either $T=a+S$ or  $T=G\setminus (-a-2S)$.
  \item [(ii)]  $S$ and $T$ are $(r,\mu-1)$--progressions for some $r$.
  \item [(iii)] $\mu =0$ and $\{S,T\}$ is an $H$--essential pair.
  \item [(iv)]There exist a subgroup $H$ and two    $H$--decompositions  $S=\bigcup \limits_{0\le i\le u}S_i$ and $T=\bigcup \limits_{0\le i\le t}T_i$  ($H$--progressions with a same difference if $\min \{|\phi(S)|, |\phi(T)|,  |\phi(G)|-|\phi(S+T)|\}\ge 2$)
  such that  one of the sets $S\setminus S_u$, $T\setminus T_{t}$ is $H$--periodic and the other is $(H,-\nu)$--periodic, and $|T_{t}+S_u|=|T_{t}|+|S_u|-\nu-\mu,$ where $0 \le \nu \le 1-\mu$.
 Moreover
  $|\phi(S+T)|=|\phi(S)|+|\phi(T)|-1,$
 \end{itemize}
  where    $\phi:G\mapsto G/H$ is the canonical map and  $H$ is a super-atom of $S$ or $T^S$ if $|G|\neq 12.$
\end{theorem}

\begin{proof}
The result holds by Theorem \ref{n-3} and Lemma \ref{modstr} if $\mu =1.$ Assume tat $\mu =0.$
The result holds by Theorem \ref{n-3} and Lemma \ref{modstr} if $|T|\ge 4$ and $|S+T|\le |G|-4$.
Assume first $|T|=3.$ By Lemma \ref{3plus3}, either (ii) holds or $T=a+S$, for some $a$. Assume now that $|T|\ge 4$ and  that $|T^S|=3.$
By Lemma \ref{leee}, $|T^S-S|=|T^S|+|S|-\zeta, $ for some $0\le \zeta \le 1.$


Suppose that one of the sets   $T^S$ and $S$  is an $r$--progression (and therefore  the other is a near--$r$--progression). Thus $R-S$ is an $r$--progression
 and the result holds. Otherwise by Lemma \ref{3plus3},
 there is an $a$ such that $R=-a-S$. Hence
$T=G\setminus (R-S)=R=G\setminus (-a-2S)$.\end{proof}

A partition $A = A_1 \cup A_0$ is said to be a  {\em quasi--$H$--periodic } partition if $A_0+H=A_0$ and $A_1$ is contained in some $H$--coset.

T

\begin{corollary}(Kemperman Structure Theorem \cite{davkem})\label{KST}
Let A and B be finite subsets of an abelian group $G$  such that  $|A + B| =
|A| + |B|-1\le |G|-2$ and   $A + B$ is aperiodic.
Then there are  a subgroup $H$ and quasi--$H$--periodic partitions  $A = A_0 \cup A_1$ and $B = B_0 \cup B_1$
  such that  $|B_{1}+A_1|=|A_{1}|+ |B_1|-1.$ Moreover $|\phi (A+B)|=|\phi (A)|+|\phi (B)|-1$ and   $|\phi (A_1+B_1-A)\cap \phi (B)|=1$, where $\phi:G\mapsto G/H$ is the canonical map.
\end{corollary}

\begin{corollary}(Grynkiewicz Structure Theorem \cite{davkem})\label{davidkk}
Let A and B be finite subsets of an abelian group $G$ such that  $|A + B| =
|A| + |B|\le |G|-3,$    $A + B$ is aperiodic.
$3\le |A|\le  |B|.$
Then one of the following holds:

\begin{itemize}
\item[(1)]$|A|=3$ and there is an $a$ such that  either $B=a+A$ or  $T=G\setminus (-a-2A)$.
\item[2)] There exist $a, b \in  G$ such that
$|(A \cup \{a\}) + (B \cup \{b\})| = |A \cup \{a\}| + |B \cup \{b\}| - 1$.
\item[3)] There exist a subgroup $H$ and  quasi--$H$--periodic partitions  $A = A_0 \cup A_1$ and $B = B_0 \cup B_1$
 such that  $|B_{1}+A_1|=|A_{1}|+ |B_1|-1.$ Moreover $|\phi (A+B)|=|\phi (A)|+|\phi (B)|-1$ and   $|\phi (A_1+B_1-A)\cap \phi (B)|=1$, where $\phi:G\mapsto G/H$ is the canonical map.
  \item [4)]  $\{A,B\}$ is a Klein pair.
\end{itemize}
\end{corollary}
The result follows easily from Theorem \ref{n-3} after two observations:
\begin{enumerate}
  \item Near-progressions  and   essential non Klein pairs  satisfy 1).
  \item  In Theorem \ref{n-3},(iv), with $\nu =1,$  satisfies also 1).
\end{enumerate}

 Without loss of generality we may take $\subgp{A^*\cup B^*}=G.$ Thus Theorem \ref{n-3} implies the last two results 
 and shows moreover that $\phi(A)$, $\phi(B)$ are progressions with the same difference
if $\min \{|\phi(A)|, |\phi(B)|,  |G|-|\phi(A+B)|\}\ge 2.$  This information is crucial in order to obtain
Lev's result \cite{levkemp} and
 Lev's type reconstructions for $|A + B| =
|A| + |B|.$  Another reconstruction follows directly by Theorem \ref{n-3}.

\end{document}